\documentclass{aims}
\usepackage{amsmath}
  \usepackage{paralist}
  \usepackage{graphics} 
  \usepackage{epsfig} 
\usepackage{graphicx}  \usepackage{epstopdf}
 \usepackage[colorlinks=true]{hyperref}
\hypersetup{urlcolor=blue, citecolor=red}
\usepackage{braket}

\newcommand\com[1]{#1}

\def\currentvolume{X}
 \def\currentissue{X}
  \def\currentyear{200X}
   \def\currentmonth{XX}
    \def\ppages{X--XX}
     \def\DOI{10.3934/xx.xx.xx.xx}

\newtheorem{theorem}{Theorem}[section]

\newtheorem{lemma}[theorem]{Lemma}

\theoremstyle{definition}
\newtheorem{definition}[theorem]{Definition}
\newtheorem{remark}{Remark}

\title[Approx. Friedrichs' syst. convex const.] 
      {Relaxation approximation of Friedrichs' systems under convex constraints}

\author[J.-F. Babadjian, C. Mifsud and N. Seguin]{}

\subjclass{Primary: 35L45, 35L60; Secondary: 35A35.}
 \keywords{Friedrichs' systems, Convex constraints.}

 \email{jean-francois.babadjian@upmc.fr}
 \email{mifsud@ljll.math.upmc.fr}
 \email{seguin@upmc.fr}

\begin{document}
\maketitle

\centerline{\scshape Jean-Fran\c cois Babadjian, Cl{\'e}ment  Mifsud}
\medskip
{\footnotesize
 \centerline{ Sorbonne Universit{\'e}s, UPMC Univ Paris 06, UMR 7598}
   \centerline{Laboratoire Jacques-Louis Lions}
   \centerline{F-75005, Paris, France }
} 

\medskip

\centerline{\scshape Nicolas Seguin}
\medskip
{\footnotesize
\centerline{ Sorbonne Universit{\'e}s, UPMC Univ Paris 06, UMR 7598}
   \centerline{Laboratoire Jacques-Louis Lions}
   \centerline{F-75005, Paris, France }
   \centerline{CNRS, UMR 7598, Laboratoire Jacques-Louis Lions}
   \centerline{F-75005, Paris, France }
   \centerline{INRIA-Paris-Rocquencourt, EPC Ange}
   \centerline{Domaine de Voluceau, BP105, 78153 Le Chesnay Cedex}
}

\bigskip

 \centerline{(Communicated by the associate editor name)}

\begin{abstract}
\com{This paper is devoted to present an approximation of a Cauchy problem for Friedrichs' systems under convex constraints. It is proved the strong convergence in $L^2_{\text{loc}}$ of a parabolic-relaxed approximation towards the unique constrained solution. }
\end{abstract}

\section{Introduction}

The aim of this paper is to prove the convergence of a relaxation approximation \com{of weak solutions to} Friedrichs' systems under \com{convex constraints. The well-posedness has been established in \cite{DLS} by means of a numerical scheme. We present here} another way to construct such weak solutions thanks to a model that relaxes the constraints. We consider the following Cauchy problem: \com{find $W:[0,T] \times \mathbb R^n \to \mathbb R^m$ such that}

\begin{equation}
\label{SystNonBorn}
\begin{cases}
\partial_t W + \sum_{j=1}^n B_j \partial_j W = 0& \text{in } ]0,T]\times \mathbb{R}^n, \\
W(t,x)\in K& \text{if } (t,x)\in [0,T]\times \mathbb{R}^n, \\
W(0,x)=W^0(x)&\text{if } x\in \mathbb{R}^n,
\end{cases}
\end{equation}
where \com{$K$ is a fixed (\textit{i.e.} independent of the time and space variables) non empty closed and convex subset of $\mathbb R^m$} containing $0$ in its interior, the matrices $B_j$ \com{are $m \times m$} symmetric matrices independent of time and space, and $T>0$.
\com{The main difficulty is due to the constraints which introduce nonlinear effects to the linear Friedrichs' system~\cite{Fried}.} This type of hyperbolic problems has been introduced in~\cite{DLS} \com{where a notion of weak solutions to  problem~\ref{SystNonBorn} has been defined}.
\begin{definition}
\label{defsolsystnonborn}
\com{Let $W^0\in L^2(\mathbb{R}^n,K)$, and $T>0$.} A function $W$ is a weak constrained solution of~\ref{SystNonBorn} if $W\in L^2([0,T]\times \mathbb{R}^n,K)$ satisfies
\begin{multline}
\label{inequadefsolsystnonborn}
\int_0^T\int_{\mathbb{R}^n} \Big( |W-\kappa|^2\partial_t \phi + \sum_{j=1}^n \Braket{W-\kappa; B_j(W-\kappa)}\partial_j \phi \Big) \, d t \, dx \\
\hspace{6.5cm}+ \int_{\mathbb{R}^n} |W^0(x)-\kappa|^2\phi(0,x)\, d x \ge 0,
\end{multline}
for all  $\kappa \in K$ and \com{$\phi \in \mathcal{C}^{\infty}_c([0,T[\times \mathbb{R}^n)$ with $\phi(t,x)\ge 0$ for all $(t,x)\in [0,T[\times \mathbb{R}^n$.}
\end{definition}
We recall here the main result of~\cite{DLS}.

\begin{theorem}
Assume that $W^0\in L^2(\mathbb{R}^n,K)$. There exists a unique weak constrained solution $W\in L^2([0,T]\times \mathbb{R}^n,K)$ to~\ref{SystNonBorn} in the sense of Definition~\ref{defsolsystnonborn}. In addition, this solution belongs to $\mathcal{C}([0,T],L^2(\mathbb{R}^n,K))$, and if \com{further} $W^0 \in H^1(\mathbb{R}^n,K)$, then $W\in L^\infty([0,T],H^1(\mathbb{R}^n,K))$.
\end{theorem}

\com{As already mentioned, the well-posedness of problem~\ref{SystNonBorn} has been established in~\cite{DLS}} thanks to a numerical method. \com{We relax here }the constraints $W(t,x)\in K$ for a.e.  $(t,x)\in ]0,T[\times \mathbb{R}^n$ as
\begin{equation}
\label{relaxed_eq_intro}
\partial_t W_\epsilon +\sum_{j=1}^n B_j \partial_{x_j} W_\epsilon = \frac{P_K(W_\epsilon)-W_\epsilon}{\epsilon},
\end{equation}
where $P_K$ \com{denotes the orthogonal }projection onto the closed convex \com{set $K$, and $\epsilon>0$ is a small parameter}. Formally, if we multiply equation~\ref{relaxed_eq_intro} by $\epsilon$, and let $\epsilon$ tend to $0$, we get that the ``limit'' of $W_\epsilon$, denoted by $W$, satisfies $P_K(W)=W$, which ensures that $W\in K$. \com{In addition, Definition~\ref{defsolsystnonborn} has been motivated in~\cite{DLS}  by a formal derivation from the relaxation system~\ref{relaxed_eq_intro}.} To see it, it suffices to take the scalar product \com{ of equation of \eqref{relaxed_eq_intro} with $W_\epsilon-\kappa$, where $\kappa\in K$ is arbitrary. We then use the first order characterization of the projection which ensures that the right hand side is non-positive, }to get the inequality of Definition~\ref{defsolsystnonborn}. The purpose of this work is to \com{rigorously} justify these formal steps.

\com{ The relaxation model presented here is very similar to viscous approximation of constrained models found in mechanics, and especially in plasticity. We start from the system of dynamical linear elasticity in three space dimension which can be written as
\begin{equation}
\left\{
\begin{array}{rcl}
\partial_t F + \nabla_x v &=& 0,\\
\partial_t v + \text{div } \sigma &=& 0,
\end{array}
\right.
\end{equation}
for all $t\in [0,T]$ and $x\in \mathbb{R}^3$. In the previous system, $F(t,x)$ is a $3\times 3$ matrix which stands for the displacement gradient, $v(t,x)\in \mathbb{R}^3$ is the velocity, {\it i.e.} the displacement time derivative, and $\sigma = \mu \left( F+F^T \right) + \lambda \left( \text{tr } F \right) I_3$ is the symmetric Cauchy stress tensor (here $\lambda$ and $\mu$ are \textit{the Lam{\'e} coefficients}, and $I_3$ is the idendity matrix in $\mathbb R^3$). This system can be rewritten (thanks to a change of variables - see~\cite{SerreMorando3D}) in the Friedrichs' framework as
$$\partial_t U + A_1 \partial_{x_1} U + A_2 \partial_{x_2} U  + A_3 \partial_{x_3} U = 0,$$
where $U$ is a vector in $\mathbb R^9$ \com{(containing the three compononents of the velocity $v$, and the six components of the stress $\sigma$)} and $A_1,A_2$ and $A_3$ are symmetric matrices. 

We now introduce the convex contraint coming from plasticity, see~\cite{SuquetDyn}. Indeed, the theory of perfect plasticity is characterized by the fact that the stress tensor $\sigma$ is constrained to stay inside a fixed closed convex set $K$ of symmetric $3 \times 3$ matrices. The total strain is then additively decomposed as the sum of (i) the elastic strain, denoted by $e$, which is still related to the stress by the linear relation $\sigma=\lambda {\rm tr} e + 2\mu e$; (ii) and the plastic strain, denoted by $p$, whose rate is oriented in a normal direction to $K$ at $\sigma$. Summarizing, one has
\begin{equation}\label{perfect_plasticity_constraint}
\frac{F+F^T}{2}=e+p, \; \sigma=\lambda {\rm tr} e + 2\mu e \in K,\; \partial_t p \in \partial I_K(\sigma),
\end{equation}
where $\partial I_K(\sigma)$ denotes the subdifferential of $I_K$, the indicator function of $K$, at the point $\sigma$. Using Fenchel-Moreau regularization of $I_K$ (see~\cite{Moreau}), the last condition in \ref{perfect_plasticity_constraint} can be relaxed as
$$\partial_t p = \frac{1}{\epsilon} \left( P_K(\sigma)-\sigma \right),$$
where $\epsilon>0$ is a viscosity parameter. We can now reformulate, at least formally, the dynamical problem of visco-plasticity (see~\cite{NouriRascle,SuquetDyn}) as 
\begin{equation}
\label{Transformation_plasticity_friedrichs}
\begin{cases}
\displaystyle \partial_t U+ A_1 \partial_{x_1} U + A_2 \partial_{x_2} U  + A_3 \partial_{x_3} U = \frac{P_{\tilde K}(U)-U}{\epsilon}& \text{on } ]0,T]\times \mathbb{R}^3, \\
U(t,x)\in \tilde{K}& \text{if } (t,x)\in [0,T]\times \mathbb{R}^3, \\
U(0,x)=U^0(x)&\text{if } x\in \mathbb{R}^3,
\end{cases}
\end{equation}
where, again, $U \in \mathbb R^9$, and $\tilde{K} = \left\{ u\in \mathbb{R}^9: \sigma \in K \right\}$, and $A_1$, $A_2$ and $A_3$ are the same matrices than in the elasto-dynamic case. As $\epsilon$ tends to zero, one expects the solution to \ref{Transformation_plasticity_friedrichs} to converge to that of the model of perfect plasticity (see \cite{Suquet} in the quasistatic case).}

%
%

\medskip

\noindent {\bf Notation.} In the sequel, we denote by $\Braket{|}$ the scalar product of $L^2(\mathbb{R}^n,\mathbb{R}^m)$ and by $\Braket{;}$ the canonical scalar product of $\mathbb{R}^m$ (and $|.|$ the associated norm). Also, to shorten notation, we write $L^2_{t,x}$ (resp. $H^1_{t,x}$) instead of $L^2(0,T;L^2(\mathbb{R}^n,\mathbb{R}^m))$ (resp. $H^1((0,T)\times \mathbb{R}^n,\mathbb{R}^m)$), $L^2_x$ instead of $L^2(\mathbb{R}^n,\mathbb{R}^m)$.

\medskip

This paper is organized as follows. In the first section, we establish the existence and uniqueness of the relaxation model thanks to a \com{parabolic approximation}. In the second section, we prove that the relaxed solution $W_\epsilon$ to~\ref{relaxed_eq_intro} satisfies the inequalities of Definition~\ref{defsolsystnonborn}. Finally, to get the existence of a solution as a limit when $\epsilon$ tends to zero of relaxed solutions $W_\epsilon$, we prove the strong convergence of the sequence $(W_\epsilon)_{\epsilon>0}$ in the space $L^2((0,T)\times \omega)$, where $\omega$ is a open bounded subset of \com{$\mathbb R^n$}, to a weak solution to the contrained Friedrichs' systems.

\section{Parabolic approximation}
In order to find a solution to the relaxation problem~\ref{relaxed_eq_intro}, \com{we use a parabolic type regularization. To this aim, we consider a classical sequence of mollifiers in $\mathbb R^n$, denoted by $(\rho_\eta)_{\eta>0}$.}
\begin{theorem}
\label{theoWepseta}
\com{ Let $W^0\in H^1(\mathbb{R}^n,K)$. For every $\epsilon>0$ and $\eta>0$, }the system
\begin{equation}
\label{eq_visco_non_borne}
\begin{cases}
\displaystyle \partial_t W_{\epsilon,\eta} - \eta \Delta W_{\epsilon,\eta} + \sum_{j=1}^n B_j \partial_j W_{\epsilon,\eta} = \frac{P_K(W_{\epsilon,\eta})-W_{\epsilon,\eta}}{\epsilon},& \text{on } ]0,T]\times \mathbb{R}^n, \\
W_{\epsilon,\eta}(0,x)=W^0(x)* \rho_\eta,&\text{if } x\in \mathbb{R}^n,
\end{cases}
\end{equation}
admits \com{a unique solution $W_{\epsilon,\eta}$  with the following properties:}
\[
W_{\epsilon,\eta}\in L^2(0,T;H^2(\mathbb{R}^n,\mathbb{R}^m)), \quad \partial_t W_{\epsilon,\eta} \in L^2(0,T;H^1(\mathbb{R}^n,\mathbb{R}^m)),
\] 
and 
$$\partial_{tt} W_{\epsilon,\eta} \in L^2(0,T;H^{-1}(\mathbb{R}^n,\mathbb{R}^m)).$$
Furthermore, we have the following estimates
\com{\begin{eqnarray}
\sup_{0\le t \le T} \left\| W_{\epsilon,\eta}(t) \right\|_{L^2_x}^2 & \le & \left\| W^0 \right\|_{L^2_x}^2,\label{estim_theo_1}\\
\sup_{0\le t \le T} \left\| W_{\epsilon,\eta} (t)\right\|_{H^1_x} & \le & C_\epsilon \left\| W^0 \right\|_{H^1_x},\label{estim_theo_2}\\
\sup_{0\le t \le T} \left\| \partial_t W_{\epsilon,\eta} (t)\right\|_{L^2_x} & \le & C_{\epsilon} \left\| W^0 \right\|_{H^1_x},\label{estim_3}
\end{eqnarray}
for some constant $C_\epsilon>0$ independent of $\eta$.}
\end{theorem}

\begin{proof}
The proof essentially follows that of Theorem 1, Part II, Section 7.3.2 in \cite{Evans}. Let $X = L^\infty(0,T;H^1(\mathbb{R}^n,\mathbb{R}^m))$ and $V\in X$. We consider the problem
\begin{equation}
\label{Visco_pt_fixe}
\begin{cases}
\displaystyle \partial_t U - \eta \Delta U  = \frac{P_K(V)-V}{\epsilon}- \sum_{j=1}^n B_j \partial_j V,& \text{on } ]0,T]\times \mathbb{R}^n, \\
U(0,x)=W^0(x)* \rho_\eta,&\text{if } x\in \mathbb{R}^n.
\end{cases}
\end{equation}
Since $0\in K$, we have the following inequality
\[
\forall k\in \mathbb{R}^m,\qquad |(P_K(k)-k)| \le |k|
\]
which shows that $\frac{P_K(V)-V}{\epsilon}- \sum_{j=1}^n B_j \partial_j V \in L^2(0,T;L^2(\mathbb{R}^n,\mathbb{R}^m))$. Using the theory of parabolic equations, we get that equation~\ref{Visco_pt_fixe} admits a unique solution $U$ with $U\in L^2(0,T;H^2(\mathbb{R}^n,\mathbb{R}^m))$ and $\partial_t U \in L^2(0,T;L^2(\mathbb{R}^n,\mathbb{R}^m))$. 
Let $\tilde{V}\in X$ and $\tilde{U}$ be the solution to~\ref{Visco_pt_fixe} associated with $\tilde{V}$. The function $\hat{U}=U-\tilde{U}$ is a solution to
\begin{equation}
\begin{cases}
\displaystyle \partial_t \hat{U} - \eta \Delta \hat{U}  = \frac{P_K(V)-P_K(\tilde{V)}-(V-\tilde{V})}{\epsilon}- \sum_{j=1}^n B_j \partial_j (V-\tilde{V}),& \text{on }]0,T]\times \mathbb{R}^n, \\
\hat{U}(0,x)=0,&\text{if } x\in \mathbb{R}^n.
\end{cases}
\end{equation}
Thanks to the theory of parabolic equations, we have the following estimate, we obtain the following estimate 
\begin{multline*}
\underset{0\le t \le T}{\text{ess-sup}} \left\| \hat{U}(t) \right\|_{H^1_x} \\
\le  C(\eta) \left\| \frac{P_K(V)-P_K(\tilde{V)}-(V-\tilde{V})}{\epsilon}- \sum_{j=1}^n B_j \partial_j (V-\tilde{V}) \right\|_{L^2_{t,x}} \\
\le  C(\eta)\max\left(\frac{2}{\epsilon}, \left\|B_j \right\| \right) \left\|V-\tilde{V} \right\|_{L^2_t H^1_x} \\
\le  C(\eta)\max\left(\frac{2}{\epsilon}, \left\|B_j \right\| \right) \sqrt{T} \underset{0\le t \le T}{\text{ess-sup}} \left\|V(t)-\tilde{V}(t) \right\|_{H^1_x}.
\end{multline*}
\com{Therefore, the mapping} 
\[
\psi : \begin{cases}
X \to \left\{ U\in L^2_t H^2_x \text{ and }\partial_t U \in L^2_{t,x}  \right\}\subset X \\
\hat{V} \mapsto \hat{U}
\end{cases}
\]
\com{is Lipschitz-continuous with Lipschitz constant bounded by $C(\eta)\max\left(\frac{2}{\epsilon}, \left\|B_j \right\| \right) \sqrt{T}$.}

We now divide $[0,T]$ into sub-intervals $[0,T_1],[T_1,2T_1],[2T_1,3T_1],\ldots,[NT_1,T]$ such that 
\[
C(\eta)\max\left(\frac{2}{\epsilon}, \left\|B_j \right\| \right) \sqrt{\max(T_1,T-NT_1)}<1.
\]
\com{For every $i \in \{0,\ldots N-1\}$, the Banach fixed-point Theorem \com{ensures the existence and uniqueness of a} solution to~\ref{eq_visco_non_borne} on the interval  $[iT_1,(i+1)T_1]$ (with initial condition $W^0* \rho_\eta$ if $i=0$, and $W_{\epsilon,\eta}(iT_1)$ if $i \geq 1$ obtained at the previous step). We then obtain a solution $W_{\epsilon,\eta}$ on the entire interval $[0,T]$ by gluing the solutions on each sub-intervals, so that $W_{\epsilon,\eta} \in L^2(0,T;H^2(\mathbb{R}^n,\mathbb{R}^m))$. According to the initial condition on each sub-intervals, the function $t \mapsto W_{\epsilon,\eta}(t)$ is continuous in $L^2_x$ at every $t=iT_1$, so that $\partial_t W_{\epsilon,\eta} \in L^2(0,T;L^2(\mathbb{R}^n,\mathbb{R}^m))$, and, in particular $W_{\epsilon,\eta}\in H^1(]0,T[\times \mathbb{R}^n,\mathbb{R}^m)$.}

To obtain the announced regularity, we use the following result (whose proof \com{relies on the chain rule in Sobolev spaces}). 

\begin{lemma}\label{sob}
\label{lemme_proj_H1}
Let $U \in H^1(]0,T[\times\mathbb{R}^n,\mathbb{R}^m)$. Then the function $P_K \circ U - U$ also belongs to $H^1(]0,T[\times\mathbb{R}^n,\mathbb{R}^m)$, and there exists a constant $M>0$, \com{independent of $U$,} such that
\[
\left\| P_K \circ U - U \right\|_{H^1_{t,x}} \le M\left\| U \right\|_{H^1_{t,x}}.
\]
\end{lemma}

This result \com{ensures} that $\frac{P_K(W_{\epsilon,\eta})-W_{\epsilon,\eta}}{\epsilon} \in H^1(]0,T[\times \mathbb{R}^n,\mathbb{R}^m)$ and then, \com{using again} to the regularity theory of parabolic equations, we obtain that \[
W_{\epsilon,\eta}\in L^2(0,T;H^2(\mathbb{R}^n,\mathbb{R}^m)),\qquad \partial_t W_{\epsilon,\eta} \in L^2(0,T;H^1(\mathbb{R}^n,\mathbb{R}^m)),
\] 
and 
\[
\partial_{tt} W_{\epsilon,\eta} \in L^2(0,T;H^{-1}(\mathbb{R}^n,\mathbb{R}^m)).
\]

Now we derive the estimates. We are going to use the \com{following result (see~\cite{Evans}).}
\begin{lemma}
\label{lem_der}
Let $U \in L^2(0,T,H^1(\mathbb{R}^n,\mathbb{R}^m))$ with $\partial_t U \in L^2(0,T;H^{-1}(\mathbb{R}^n,\mathbb{R}^m))$. Then, the function
\[
t\mapsto \left\| U(t) \right\|_{L^2_x}^2,
\]
is absolutely continuous, \com{ and for a.e. $t \in [0,T]$,}
\[
\frac{d}{d t} \left(  \frac{1}{2} \left\| U(t) \right\|_{L^2_x}^2 \right) = \Braket{U(t) , \partial_t U(t)}_{H^1_x,H^{-1}_x}.
\]
\end{lemma}
\com{Applying this result to $W_{\epsilon,\eta}$, we get that for a.e. $t \in [0,T]$,}
\begin{multline}\label{1}
 \frac{d}{d t} \left(  \frac{1}{2} \left\| W_{\epsilon,\eta}(t) \right\|_{L^2_x}^2\right) 
 = \Braket{ W_{\epsilon,\eta}(t)| \partial_t W_{\epsilon,\eta}(t)}_{L^2_x} \\
 = \Braket{ W_{\epsilon,\eta} (t)| \frac{P_K(W_{\epsilon,\eta})(t)-W_{\epsilon,\eta}(t)}{\epsilon}-\sum_{j=1}^n B_j \partial_j W_{\epsilon,\eta} (t)+ \eta \Delta W_{\epsilon,\eta} (t)},
\end{multline}
where we used the fact that $W_{\epsilon,\eta}$ is a solution to the partial differential equation \ref{eq_visco_non_borne}. \com{Since $0\in K$, we have}
\begin{multline}\label{2}
\Braket{W_{\epsilon,\eta}(t) |P_K(W_{\epsilon,\eta})(t)-W_{\epsilon,\eta}(t)}_{L^2_{x}} \\
= \Braket{P_K(W_{\epsilon,\eta})(t) | P_K(W_{\epsilon,\eta})(t)-W_{\epsilon,\eta}(t) } - \left\| P_K(W_{\epsilon,\eta})(t)-W_{\epsilon,\eta}(t) \right\|_{L^2_{x}}^2 \leq 0.
\end{multline}
On the other hand, if $v\in \mathcal{C}^{\infty}_c(\mathbb{R}^n,\mathbb{R}^m)$, an integration by parts shows that
\[
\Braket{v|\eta \Delta v}_{L^2} = -\eta \left\| Dv \right\|_{L^2_x}^2 \le 0,
\]
and
\[
\Braket{ v| \sum_{j=1}^n B_j \partial_j v } = 0,
\]
since the matrices $B_j$ are \com{symmetric and }independent of the space variables. By approximation, these formulas are true for $v \in H^1(\mathbb R^n,\mathbb R^m)$ as well, and in particular, 
\begin{equation}\label{3}
\Braket{ W_{\epsilon,\eta} (t)|\sum_{j=1}^n B_j \partial_j W_{\epsilon,\eta} (t)+ \eta \Delta W_{\epsilon,\eta} (t)} \leq 0.
\end{equation}
Gathering \ref{1}, \ref{2} and \ref{3}, we obtain that
\[
\frac{d}{d t} \left(  \frac{1}{2} \left\| W_{\epsilon,\eta} \right\|_{L^2_x}^2\right) \le 0,
\]
and using Gronwall's Lemma, we derive the first estimate \ref{estim_theo_1}.
\begin{equation}
\label{estim_1}
\sup_{0\le t \le T} \left\| W_{\epsilon,\eta}(t) \right\|_{L^2_x}^2 = \left\| W^0 * \rho_\eta \right\|_{L^2_x}^2 \le \left\| W^0 \right\|_{L^2_x}^2.
\end{equation}

We apply the same \com{argument} to \com{the spatial weak derivates} $\partial_k W_{\epsilon,\eta}$ of $W_{\epsilon,\eta}$. Deriving the partial differential equation \ref{eq_visco_non_borne} in the sense of distribution, we infer that
\[
\partial_t \partial_k W_{\epsilon,\eta} - \eta \Delta \partial_k W_{\epsilon,\eta} + \sum_{j=1}^n B_j \partial_j \partial_k W_{\epsilon,\eta} = \partial_k \frac{P_K(W_{\epsilon,\eta})-W_{\epsilon,\eta}}{\epsilon}.
\]
The previous equality \com{actually holds} in
$L^2(0,T,L^2(\mathbb{R}^n,\mathbb{R}^m))$ thanks to the regularity of $W_{\epsilon,\eta}$, and we can apply Lemma~\ref{lem_der} to obtain that \com{{\^E}for a.e. $t \in [0,T]$,}
\begin{multline*}
\frac{d}{d t} \left(  \frac{1}{2} \left\| \partial_k W_{\epsilon,\eta} (t)\right\|_{L^2_x}^2\right)
= \Braket{ \partial_k W_{\epsilon,\eta}(t), \partial_t \partial_k W_{\epsilon,\eta}(t)}_{L^2_x} \\
= \Braket{ \partial_k W_{\epsilon,\eta}(t) | \partial_k \frac{P_K(W_{\epsilon,\eta})(t)-W_{\epsilon,\eta}(t)}{\epsilon}-\sum_{j=1}^n B_j \partial_j \partial_k W_{\epsilon,\eta}(t) + \eta \Delta \partial_k W_{\epsilon,\eta}(t) }.
\end{multline*}
\com{Arguing as in \ref{3}, we get}
\[
\begin{aligned}
  \frac{d}{d t} \left( \frac{1}{2} \left\| \partial_k
      W_{\epsilon,\eta} (t)\right\|_{L^2_x}^2\right) &\le \Braket{
    \partial_k W_{\epsilon,\eta} (t)| \partial_k
    \frac{P_K(W_{\epsilon,\eta}(t))-W_{\epsilon,\eta}(t)}{\epsilon}}\\
  &\le \frac{M}{\epsilon} \left\| \partial_k W_{\epsilon,\eta}
    (t)\right\|_{L^2_x}^2,
\end{aligned}
\]
\com{where we used Lemma \ref{sob}.} Using again Gronwall's Lemma, it yields 
\begin{eqnarray}
\sup_{0\le t \le T} \left\| \partial_k W_{\epsilon,\eta}(t) \right\|_{L^2_x}^2 &\le& \exp\left( \frac{2TM}{\epsilon} \right) \left\| \partial_k W^0 * \rho_\eta \right\|_{L^2_x}^2 \nonumber\\
&\le& \exp\left( \frac{2TM}{\epsilon} \right) \left\| \partial_k W^0 \right\|_{L^2_x}^2,\label{estim_2}
\end{eqnarray}
\com{which completes the proof of estimate \ref{estim_theo_2}.}

We \com{finally derive the last estimate for } $\partial_{t} W_{\epsilon,\eta}$. Again, we derive the partial differential equation \ref{eq_visco_non_borne}  \com{with respect to $t$ in the distributional sense} to get 
\[
\partial_{tt} W_{\epsilon,\eta} - \eta \Delta \partial_t W_{\epsilon,\eta} + \sum_{j=1}^n B_j \partial_j \partial_t W_{\epsilon,\eta} = \partial_t \frac{P_K(W_{\epsilon,\eta})-W_{\epsilon,\eta}}{\epsilon}
\]
in $L^2(0,T,H^{-1}(\mathbb{R}^n,\mathbb{R}^m))$. \com{Using again} Lemma~\ref{lem_der}, we get that for a.e. $t \in [0,T]$,
\begin{multline*}
\frac{d}{d t} \left(  \frac{1}{2} \left\| \partial_t W_{\epsilon,\eta} (t)\right\|_{L^2_x}^2\right) = \Braket{ \partial_t W_{\epsilon,\eta}(t) | \partial_{tt} W_{\epsilon,\eta}(t)}_{H^1_x,H^{-1}_x}\\
=\Braket{ \partial_t W_{\epsilon,\eta}(t)| \eta \Delta \partial_t W_{\epsilon,\eta}(t) - \sum_{j=1}^n B_j \partial_j \partial_t W_{\epsilon,\eta}(t) + \partial_t \frac{P_K(W_{\epsilon,\eta})(t)-W_{\epsilon,\eta}(t)}{\epsilon} }_{H^1_x,H^{-1}_x}.
\end{multline*}
\com{As before, we infer that
\[
\Braket{ \partial_t W_{\epsilon,\eta}(t)| \sum_{j=1}^n B_j \partial_j \partial_t W_{\epsilon,\eta}(t)}_{L^2_x}=0, \quad
\Braket{ \partial_t W_{\epsilon,\eta}(t)| \eta \Delta \partial_t W_{\epsilon,\eta}(t) }_{H^1_x,H^{-1}_x}\le 0,
\] 
which shows, thanks to Lemma \ref{sob}, that
\begin{multline*}
\frac{d}{d t} \left(  \frac{1}{2} \left\| \partial_t W_{\epsilon,\eta} (t)\right\|_{L^2_x}^2\right)\le \frac{1}{\epsilon} \left\| \partial_t W_{\epsilon,\eta} (t)\right\|_{L^2_x} \left\| \partial_t(P_K(W_{\epsilon,\eta})(t)-W_{\epsilon,\eta}(t)) \right\|_{L^2_x} \\
\le \frac{M}{\epsilon} \left\| \partial_t W_{\epsilon,\eta}(t) \right\|_{L^2_x}^2.
\end{multline*}
At this point, we would like to use  Gronwall's Lemma. To do that, we need to know the value} of $\partial_t W_{\epsilon,\eta}$ at $t=0$. To this aim, let us take a test function $\phi \in \mathcal{C}^\infty_{c}(\mathbb{R}\times \mathbb{R}^n,\mathbb{R}^m)$ with $\phi(T,\cdot) = 0$. On the one hand, \com{according to Fubini's Theorem and Green's formula on $[0,T]$,} {\^E}we have
\[
\int_0^T \int_{\mathbb{R}^n} \Braket{\partial_t W_{\epsilon,\eta} ; \partial_t \phi } = \int_{\mathbb{R}^n} \left[ - \Braket{\partial_t W_{\epsilon,\eta}(0,\cdot);\phi(0,\cdot)} -\int_0^T  \Braket{\partial_{tt} W_{\epsilon,\eta} ; \phi}  \right],
\] 
\com{since $\partial_t W_{\epsilon,\eta} \in H^1(0,T;H^{-1}(\mathbb{R}^n,\mathbb{R}^m))$ and $\phi$ is smooth.} On the other hand, \com{according to equation \ref{eq_visco_non_borne}}
\begin{multline*}
\int_0^T \int_{\mathbb{R}^n} \Braket{\partial_t W_{\epsilon,\eta} ; \partial_t \phi } \\
 =  \int_0^T \int_{\mathbb{R}^n} \Braket{\eta \Delta W_{\epsilon,\eta} - \sum_{j=1}^n B_j \partial_j W_{\epsilon,\eta} + \frac{P_K(W_{\epsilon,\eta})-W_{\epsilon,\eta}}{\epsilon} ; \partial_t \phi } \\
 =  \int_{\mathbb{R}^n} \left\{ - \int_0^T \Braket{\partial_t\left[ \eta\Delta W_{\epsilon,\eta} - \sum_{j=1}^n B_j \partial_j W_{\epsilon,\eta} + \frac{P_K(W_{\epsilon,\eta})-W_{\epsilon,\eta}}{\epsilon} \right] ; \phi}  \right. \\
   - \left. \Braket{\eta\Delta W^0*\rho_\eta - \sum_{j=1}^n B_j \partial_j W^0*\rho_\eta + \frac{P_K(W^0*\rho_\eta)-W^0*\rho_\eta}{\epsilon} ; \phi(0,\cdot)} \right\}.
\end{multline*}

\noindent Since $\phi(0,\cdot)$ is arbitrary, we obtain that 
\[
\partial_t W_{\epsilon,\eta}(0,\cdot) = \eta\Delta W^0*\rho_\eta - \sum_{j=1}^n B_j \partial_j W^0*\rho_\eta + \frac{P_K(W^0*\rho_\eta)-W^0*\rho_\eta}{\epsilon}.
\]
\com{We are now in position to apply Gronwall's Lemma which implies that}
\[
\sup_{0\le t \le T} \left\| \partial_t W_{\epsilon,\eta} (t)\right\|_{L^2_x}^2 \le \exp\left(\frac{2TM}{\epsilon} \right) \left\| \partial_t W_{\epsilon,\eta}(0,\cdot) \right\|_{L^2_x}^2.
\]
\com{Since,
\[
\left\| \Delta W^0*\rho_\eta \right\|_{L^2_x} \le \frac{C}{\eta} \left\| \nabla W^0 \right\|_{L^2_x},
\]
for some constant $C>0$ independent of $\eta$, we deduce that
$$\sup_{0\le t \le T} \left\| \partial_t W_{\epsilon,\eta} (t)\right\|_{L^2_x} \le C_\epsilon \left\| W^0 \right\|_{H^1_x},$$
where $C_\epsilon>0$ is another constant independent of $\eta$, which completes the proof of the last estimate \ref{estim_3}.}
\end{proof}

\section{Approximation of the convex constraints}
We now consider the relaxation problem
\begin{equation}
\label{ApproxSystNonBorn}
\begin{cases}
\displaystyle \partial_t W_{\epsilon} + \sum_{j=1}^n B_j \partial_j W_{\epsilon} = \frac{P_K(W_{\epsilon})-W_{\epsilon}}{\epsilon},& \text{on } ]0,T]\times \mathbb{R}^n, \\
W_{\epsilon}(0,x)=W^0(x),&\text{if } x\in \mathbb{R}^n.
\end{cases}
\end{equation}
Thanks to Theorem \ref{theoWepseta}, we will construct the solution to the previous problem as the limit of the solution of the parabolic problem~\ref{eq_visco_non_borne} when $\eta$ tends to zero.
\begin{theorem}
\com{There exists a unique solution $W_\epsilon \in H^1(]0,T[ \times \mathbb{R}^n,\mathbb{R}^m))$ to~\ref{ApproxSystNonBorn} satisfying, for all $\phi \in \mathcal{C}^\infty_c(\mathbb{R}\times \mathbb{R}^n, \mathbb{R}^m)$,}
\begin{equation}\label{cons}
 \int_0^T \int_{\mathbb{R}^n} \!\Braket{\partial_t W_{\epsilon} + \sum_{j=1}^n B_j \partial_j W_{\epsilon};\phi} d x \,d t 
  = \int_0^T \int_{\mathbb{R}^n} \!\Braket{\frac{P_K(W_{\epsilon})-W_{\epsilon}}{\epsilon};\phi} d x \,d t,
  \end{equation}
and $W_\epsilon(0,\cdot) = W_0$ in $L^2(\mathbb{R}^n,\mathbb{R}^m)$. In addition,
\begin{equation}
\label{estim_Weps_LinfL2}
\sup_{0\le t \le T} \left\| W_{\epsilon}(t) \right\|_{L^2_x}^2 \le \left\| W^0 \right\|_{L^2_x}^2.
\end{equation}
\end{theorem}

\begin{proof}
Thanks to Theorem~\ref{theoWepseta}, the sequence $(W_{\epsilon,\eta})_{\eta>0}$ is bounded in the space $H^1(]0,T[\times \mathbb{R}^n,\mathbb{R}^m)$. We can thus extract a subsequence \com{(not relabeled) such that}
$$W_{\epsilon,\eta} \rightharpoonup  W_\epsilon  \text{ weakly in } H^1(]0,T[\times \mathbb{R}^n,\mathbb{R}^m). $$
\com{In particular, \ref{estim_Weps_LinfL2} is a consequence of \ref{estim_theo_1} by the lower semicontinuity of the norm with respect to weak convergence. Since the embedding of $H^1(]0,T[\times \mathbb{R}^n,\mathbb{R}^m)$ into $L^2_{\rm loc}([0,T]\times \mathbb{R}^n,\mathbb{R}^m) $ is compact (cf~\cite{Adams}), we deduce that
$$W_{\epsilon,\eta} \to  W_\epsilon  \text{ strongly in } L^2(]0,T[\times ]-R,R[^n,\mathbb{R}^m),$$
for each $R>0$. } Let $\phi \in \mathcal{C}^\infty_c (\mathbb{R}\times \mathbb{R}^n,\mathbb{R}^m)$ and $R$ such that the support of $\phi$ is \com{contained} in $[-R,R]^{n+1}$. Since $W_{\epsilon,\eta}$ is a weak solution of \ref{Visco_pt_fixe}, we have
\begin{multline*}
\int_0^T \int_{\mathbb{R}^n} \Big(\Braket{ \partial_t W_{\epsilon,\eta}; \phi} + \eta \sum_{j=1}^n\Braket{ \partial_jW_{\epsilon,\eta} ;\partial_j \phi} +\sum_{j=1}^n \Braket{B_j \partial_j W_{\epsilon,\eta} ; \phi}\Big)\, d x \,d t \\
= \int_0^T \int_{\mathbb{R}^n} \Braket{\frac{P_K(W_{\epsilon,\eta})-W_{\epsilon,\eta}}{\epsilon};\phi} \, d x\, d t.
\end{multline*}
Using the weak convergence, we infer that
\begin{multline*}
\int_0^T \int_{\mathbb{R}^n} \Big( \Braket{ \partial_t W_{\epsilon,\eta}; \phi} +  \eta \sum_{j=1}^n\Braket{ \partial_jW_{\epsilon,\eta} ;\partial_j \phi} +\sum_{j=1}^n \Braket{B_j \partial_j W_{\epsilon,\eta} ; \phi} \Big)\, d x \,d t \\
\to \int_0^T \int_{\mathbb{R}^n} \Big(\Braket{ \partial_t W_{\epsilon}; \phi} + \sum_{j=1}^n \Braket{B_j \partial_j W_{\epsilon} ; \phi} \Big)\, d x\, d t.
\end{multline*}
On the other hand, the strong convergence yields
\[
\Braket{\frac{P_K(W_{\epsilon,\eta})-W_{\epsilon,\eta}}{\epsilon};\phi}
\to  \Braket{\frac{P_K(W_{\epsilon})-W_{\epsilon}}{\epsilon};\phi}\text{ strongly in } L^1 (]0,T[ \times \mathbb R^n), 
\]
and consequently, we obtain that
\begin{equation}
\label{equa_sans_visc}
\int_0^T \int_{\mathbb{R}^n}\!\Braket{\partial_t W_{\epsilon} + \sum_{j=1}^n B_j \partial_j W_{\epsilon};\phi} d x\, d t 
  = \int_0^T \int_{\mathbb{R}^n} \!\Braket{\frac{P_K(W_{\epsilon})-W_{\epsilon}}{\epsilon};\phi} d x \,d t.
\end{equation} 
We next focus on the initial condition. We take $\phi \in \mathcal{C}^\infty_c (]-\infty,T[\times \mathbb{R}^n,\mathbb{R}^m)$ (in particular $\phi(T)=0$). An integration by parts shows that
\begin{multline*}
-\int_0^T \int_{\mathbb{R}^n}  \Braket{ W_{\epsilon,\eta}; \partial_t \phi}\, d x \, d t + \int_{\mathbb{R}^n} \Braket{ W^0*\rho_{\eta}(x); \phi(0,x)} d x \\
= \int_0^T \int_{\mathbb{R}^n} \Braket{ \partial_t W_{\epsilon,\eta}; \phi}\, d x \, d t.
\end{multline*}
Letting $\eta$ tend to zero, and using \ref{equa_sans_visc} leads to
\begin{multline}
\label{green_sur_visc}
\int_0^T \int_{\mathbb{R}^n} \Big( - \Braket{ W_{\epsilon}; \partial_t \phi} +\sum_{j=1}^n \Braket{B_j \partial_j W_{\epsilon} ; \phi}\Big) \,d x  \, d t \\
 + \int_{\mathbb{R}^n} \Braket{ W^0(x); \phi(0,x)} d x  = \int_0^T \int_{\mathbb{R}^n} \Braket{\frac{P_K(W_{\epsilon})-W_{\epsilon}}{\epsilon};\phi}d x \, d t,
\end{multline}
since $W^0*\rho_{\eta} \to W^0$ strongly in $L^2_{\text{loc}}(\mathbb R^n,\mathbb R^m)$. We now integrate by parts in~\ref{equa_sans_visc}, using the fact that $W_\epsilon \in H^1(0,T,L^2(\mathbb{R}^n,\mathbb{R}^m))$,
\begin{multline}
\label{green_sur_final}
\int_0^T \int_{\mathbb{R}^n} \Big( \Braket{ W_{\epsilon}; \partial_t \phi} +\sum_{j=1}^n \Braket{B_j \partial_j W_{\epsilon} ; \phi}\Big)\,d x \, d t \\ 
+ \int_{\mathbb{R}^n} \Braket{ W_\epsilon(0,x); \phi(0,x)} d x  = \int_0^T \int_{\mathbb{R}^n} \Braket{\frac{P_K(W_{\epsilon})-W_{\epsilon}}{\epsilon};\phi} d x\, d t.
\end{multline}
Using the equations~\ref{green_sur_visc} and~\ref{green_sur_final}, it gives that
\[
\int_{\mathbb{R}^n} \Braket{ W_\epsilon(0,x); \phi(0,x)} d x = \int_{\mathbb{R}^n} \Braket{ W^0(x); \phi(0,x)} d x,
\]
and consequently the initial condition is satisfied in $L^2(\mathbb{R}^n,\mathbb{R}^m)$. 

\com{It remains to show the uniqueness. Let us consider two solutions} $W_\epsilon$ and $\tilde{W_{\epsilon}}$ associated with the same initial condition $W^0$. Using the partial differential equation \ref{ApproxSystNonBorn}, we obtain that
\[
\partial_t \left( W_\epsilon-\tilde{W_{\epsilon}} \right)+ \sum_{j=1}^n B_j \partial_j \left(W_\epsilon-\tilde{W_{\epsilon}} \right) = \frac{P_K(W_{\epsilon})-W_\epsilon-P_K(\tilde{W_{\epsilon}})+\tilde W^\epsilon}{\epsilon}.
\]
As already observed, we know that for a.e. $t \in [0,T]$,
\[
\int_{\mathbb{R}^n} \sum_{j=1}^n \Braket{B_j \partial_j W_\epsilon(t)-\tilde{W_{\epsilon}}(t) ; W_\epsilon(t)-\tilde{W_{\epsilon}}(t)} d x = 0,
\]
and also
\[
\frac{1}{2}\frac{d }{d  t} \left\| W_\epsilon(t)-\tilde{W_{\epsilon}}(t) \right\|_{L^2_x}^2 =  \Braket{ \partial_t W_\epsilon(t)-\partial_t \tilde{W_{\epsilon}}(t)| W_\epsilon(t)-\tilde{W_{\epsilon}}(t)}. 
\]
Consequently, we have that
\begin{multline*}
\frac{1}{2} \frac{d }{d  t} \left\| W_\epsilon(t)-\tilde{W_{\epsilon}} (t)\right\|_{L^2_x}^2 \\
=\Braket{\frac{P_K(W_{\epsilon})(t)-W_\epsilon(t)-P_K(\tilde{W_{\epsilon}})(t)+\tilde W_\epsilon(t)}{\epsilon}|W_\epsilon(t)-\tilde{W_{\epsilon}}(t)}\\
 \le \frac{2}{\epsilon} \left\| W_\epsilon(t) - \tilde{W_{\epsilon}}(t) \right\|_{L^2_x}^2,
\end{multline*}
since the projection is $1$-Lipschitz. Gronwall's Lemma thus implies that $W_\epsilon = \tilde{W_\epsilon}$ since they satisfy the same initial condition. \com{As a consequence of the uniqueness, we deduce that there is no need to extract a subsequence from $(W_{\epsilon,\eta})_{\eta>0}$ to get the convergences as $\eta \to 0$.}
\end{proof}

\section{Convergence of the \com{relaxed formulation}}
In this section, we first show that the solution $W_\epsilon$ to the relaxation problem~\ref{ApproxSystNonBorn} \com{satisfies} the inequality of Definition~\ref{defsolsystnonborn}, and then we prove that we can pass to the limit in this inequality to get a solution to the \com{initial} problem~\ref{SystNonBorn}.
\begin{lemma}
\label{lemma_Weps_ineq}
Let $W_\epsilon$ be the unique solution to~\ref{ApproxSystNonBorn}. For all $\kappa \in K$ and for all $\phi \in W^{1,\infty}((-\infty,T)\times \mathbb{R}^n,\mathbb{R}^+)$ with compact support in $]-\infty,T[\times \mathbb{R}^n$, one has
\begin{multline}
\label{inequa_Weps}
\int_0^T\int_{\mathbb{R}^n} \Big( |W_\epsilon-\kappa|^2\partial_t \phi + \sum_{j=1}^n \Braket{W_\epsilon-\kappa; B_j(W_\epsilon-\kappa)}\partial_j \phi \Big)  d t \, d x \\
+ \int_{\mathbb{R}^n} |W^0(x)-\kappa|^2\phi(0,x)\, d x \ge 0.
\end{multline}
\end{lemma}

\begin{proof}
Since $W_\epsilon$ is a solution to~\ref{ApproxSystNonBorn}, we know that
\[
\int_0^T \int_{\mathbb{R}^n}\Braket{\partial_t W_{\epsilon} + \sum_{j=1}^n B_j \partial_j W_{\epsilon} - \frac{P_K(W_{\epsilon})-W_{\epsilon}}{\epsilon};W_\epsilon - \kappa}d x \,d t = 0.
\]
\com{By the first order characterization of the projection}, one has $$\Braket{ \frac{P_K(W_{\epsilon})-W_{\epsilon}}{\epsilon};W_\epsilon - \kappa}(t,x)\le 0$$ \com{for a.e. $(t,x) \in [0,T] \times \mathbb R^n$. On the other hand, since $W_\epsilon \in H^1(0,T;L^2(\mathbb{R}^n,\mathbb{R}^m))$ and $W_\epsilon \in L^2(0,T;H^1(\mathbb{R}^n,\mathbb{R}^m))$) we can integrate by parts to obtain the desired result.}
\end{proof}

\begin{remark}
\com{Let us stress that, although the function $W_\epsilon$ satisfies the same inequality than the weak constrained solution, it is not a weak constrained solution in the sense of Definition~\ref{defsolsystnonborn} because it does not  \textit{a priori} belong to $K$.}
\end{remark}

To get a weak constrained solution from the sequence of solutions $(W_\epsilon)_{\epsilon>0}$ to the relaxation problem~\ref{ApproxSystNonBorn}, we need to pass to the limit as $\epsilon\to 0$ in the previous inequality. This is the purpose of the following result.

\begin{theorem}
\com{For every  bounded open set $\omega \subset \mathbb{R}^n$, the sequence $(W_\epsilon)_{\epsilon>0}$  converges strongly in $L^2((0,T)\times \omega,\mathbb{R}^m)$ to some  function $W$ which is a weak constrained solution to problem~\ref{SystNonBorn}.}
\end{theorem}

\begin{proof}
Let $\omega$ be an open bounded subset of $\mathbb{R}^n$. We are going to prove \com{the existence of} a subsequence of $(W_\epsilon)_{\epsilon>0}$ (associated with the same initial data $W^0$) which converges in $L^2(0,T,L^2(\omega,\mathbb{R}^m))$. \com{We will use the following compactness criterion (see ~\cite{Simon}).}
\begin{theorem}
\label{theo_simons}
Let $B$ be a Banach space. A subset $F$ of $L^2(0,T;B)$ is relatively compact if and only if both conditions are fulfilled:
\begin{itemize}
\item the set $\left\{ \int_{t_1}^{t_2} f(t)\, d  t :\,  f\in F \right\}$ is relatively compact in $B$ for all $0<t_1<t_2<T$;
\item \com{we have
\[
\sup_{f \in F} \left\| \tau_h f - f\right\|_{L^2(0,T-h;B)} \underset{h\to 0}{\rightarrow} 0,
\]
where $\tau_h f:(t,x)\mapsto \tau_h f(t,x):=f(t+h,x)$.}
\end{itemize}
\end{theorem}

We are going to apply this result to
$F=({W_\epsilon}_{|\omega})_\epsilon$ where ${W_\epsilon}_{|\omega}$
\com{is the restriction to $[0,T] \times \omega$ of $W_\epsilon$}. We
first show that the set $\mathcal{F} = \left\{ \int_{t_1}^{t_2}
  {W_\epsilon}_{|\omega}(t,\cdot)\, d t : \, \epsilon>0 \right\}$ is
relatively compact in $L^2(\omega,\mathbb{R}^m)$ for all
$0<t_1<t_2<T$.  \com{We first observe that} $\mathcal{F}$ is bounded
in $L^2(\omega,\mathbb{R}^m)$ by $(t_2-t_1)\left\|
  W^0\right\|_{L^2_{x}}$ thanks to estimate~\ref{estim_Weps_LinfL2}.
To show that $\mathcal{F}$ is relatively compact in
$L^2(\omega,\mathbb{R}^m)$, \com{it is enough to check the validity
  the Riesz-Fr{\'e}chet-Kolmogorov compactness criterion
  (see~\cite{Brezis} remark 13 page 74), {\it i.e.},}
\begin{equation}
\label{first_limit_theo_simon}
\lim_{h\to 0} \sup_{\epsilon>0}\left\| \int_{t_1}^{t_2} (W_\epsilon(t,x+h)-W_\epsilon(t,x))\, d t \right\|_{L^2_x} =0.
\end{equation}
Note that $W_{\epsilon,h}:=W_\epsilon(\cdot,\cdot+h)$ is a solution to the problem~\ref{ApproxSystNonBorn} associated with the initial condition $W^0(\cdot+h)$. \com{Consequently, since the projection is $1$-Lipschitz, we have for all $\phi \in W^{1,\infty}((-\infty,T)\times \mathbb{R}^n,\mathbb{R}^+)$ with compact support in $]-\infty,T[\times \mathbb{R}^n$, }
\begin{multline*}
\int_0^T \int_{\mathbb{R}^n} \Braket{\partial_t (W_{\epsilon,h}-W_\epsilon) + \sum_{j=1}^n B_j \partial_j (W_{\epsilon,h}-W_\epsilon),W_{\epsilon,h}-W_\epsilon}\varphi \,d t \,d x \\
=\frac{1}{\epsilon}\int_0^T \int_{\mathbb{R}^n} \Braket{P_K(W_{\epsilon,h})-P_K(W_\epsilon)-(W_{\epsilon,h}-W_\epsilon),W_{\epsilon,h}-W_\epsilon}\varphi \, d t \, d x \leq 0,
\end{multline*}
which implies that
\begin{multline}
\label{ineq_diff_esp}
 \int_0^T \int_{\mathbb{R}^n} \Big( |W_{\epsilon,h}-W_\epsilon|^2\partial_t\varphi +\sum_{i=1}^n \Braket{W_{\epsilon,h}-W_\epsilon,B_i (W_{\epsilon,h}-W_\epsilon)}\partial_{x_i}\varphi\Big)\,  d t\,  d x \\
  + \int_{\mathbb{R}^n} |W^0-W^0_h|^2(x) \varphi(0,x)\, d x \ge 0, 
\end{multline}
\noindent where $W^0_h=W^0(\cdot+h)$. Let $r>0$, we define the function $\varphi$ as
\[
\varphi(t,x) = \left\{
\begin{array}{ll}
\frac{T-t}{T}+\frac{r-|x|}{nLT}, &\text{if } t\in [0,T]\text{ and } r \le |x| \le r+nL(T-t), \\
\frac{T-t}{T}, &\text{if } t\in [0,T]\text{ and } x\in B(0,r), \\
0, & \text{otherwise},
\end{array}
\right.
\]
where $L$ is the maximum of the spectral \com{radii} of the matrices $B_i$. \com{We claim that for a.e. $(t,x) \in [0,T]\times\mathbb{R}^n$}
\[
\left(|W_{\epsilon,h}-W_\epsilon|^2 \partial_t\varphi +\sum_{i=1}^n \Braket{W_{\epsilon,h}-W_\epsilon,B_i (W_{\epsilon,h}-W_\epsilon)}\partial_{x_i}\varphi \right)(t,x)\le 0.
\]
This inequality is obviously satisfied as soon as $x\in B(0,r)$ and $t\in [0,T]$. In the case where $r \le |x| \le r+nL(T-t)$, we get for all $1\le i\le n$
\[
\Braket{W_{\epsilon,h}-W_\epsilon,B_i (W_{\epsilon,h}-W_\epsilon)}(t,x) \ge -L|W_{\epsilon,h}-W_\epsilon|^2(t,x).
\]
Consequently multiplying by $\partial_{x_i}\varphi$, it yields
$$ \left(\sum_{i=1}^n \Braket{W_{\epsilon,h}-W_\epsilon,B_i (W_{\epsilon,h}-W_\epsilon)}\partial_{x_i}\varphi\right)(t,x) 
\le   -|W_{\epsilon,h}-W_\epsilon|^2(t,x)\partial_t \varphi(t,x),
$$
and the inequality is true also in that case. 
According to \ref{ineq_diff_esp}, we obtain that
\begin{multline*}
 \int_{\mathbb{R}^n} |W^0-W^0_h|^2(x)\varphi(0,x)\,d x \\
 \ge  -\int_0^T \int_{B(0,r)} \Big(|W_{\epsilon,h}-W_\epsilon|^2 \partial_t \varphi
+\sum_{i=1}^n \Braket{W_{\epsilon,h}-W_\epsilon,B_i (W_{\epsilon,h}-W_\epsilon)}\partial_{x_i}\varphi\Big)\,d t \, d x.
\end{multline*}
Thanks to the definition of $\varphi$, we get
\[
\int_0^T \int_{B(0,r)} |W_{\epsilon,h}-W_{\epsilon}|^2\, d x \, d t \le T \int_{B(0,r+nLT)} |W^0-W^0_h|^2\, d x,
\]
and the regularity of $W_0$ together with \cite[Proposition 9.3]{Brezis} yields
\[
\int_0^T \int_{B(0,r)} |W_{\epsilon,h}-W_{\epsilon}|^2\, d x\, d t \le T |h|^2 \int_{\mathbb{R}^n} |\nabla W^0|^2\, d x.
\]
Therefore, \ref{first_limit_theo_simon} holds, and consequently the set $\mathcal{F}$ is relatively compact in $L^2(\omega,\mathbb{R}^m)$ for all $0<t_1<t_2<T$.

\medskip

It remains to show that
\[
\lim_{h \to 0} \sup_{\epsilon>0}\left\| \tau_h W_{\epsilon} - W_\epsilon\right\|_{L^2(0,T-h;L^2(\omega,\mathbb{R}^m))}  0.
\]
For all $\phi \in W^{1,\infty}(]-\infty,T-h[\times \mathbb{R}^n,\mathbb{R}^+)$ with compact support in $]-\infty,T-h[\times \mathbb{R}^n$, one has
\begin{multline*}
\int_0^{T-h} \int_{\mathbb{R}^n} \Braket{\partial_t (\tau_h W_{\epsilon} -W_\epsilon),(\tau_h W_{\epsilon} -W_\epsilon)}\varphi\, d t \,d x \\
+  \int_0^{T-h} \int_{\mathbb{R}^n}\Braket{\sum_{j=1}^n B_j \partial_j (\tau_h W_{\epsilon} -W_\epsilon),(\tau_h W_{\epsilon} -W_\epsilon)}\varphi\, d t\, d x  \\
= \frac{1}{\epsilon}\int_0^{T-h} \int_{\mathbb{R}^n} \Braket{P_K(\tau_h W_{\epsilon} )-P_K(W_\epsilon)-(\tau_h W_{\epsilon} -W_\epsilon),\tau_h W_{\epsilon} -W_\epsilon}\varphi\, d t \, d x\leq 0,
\end{multline*}
\com{since the projection is $1$-Lipschitz. Arguing as before, we obtain}
\begin{multline*}
 \int_0^{T-h} \int_{\mathbb{R}^n} |\tau_h W_{\epsilon}-W_\epsilon|^2\partial_t\varphi\, d t\, d x \\
+ \int_0^{T-h} \int_{\mathbb{R}^n}  \sum_{i=1}^n\Braket{\tau_h W_{\epsilon}-W_\epsilon,B_i (\tau_h W_{\epsilon}-W_\epsilon)}\partial_{x_i}\varphi\, d t\, d x \\
 + \int_{\mathbb{R}^n} |W_{\epsilon}(h,x)-W^0(x)|^2 \varphi(0,x)\, d x \ge 0.
\end{multline*}

\noindent Using a similar test function 
\[
\varphi(t,x) = \left\{
\begin{array}{ll}
\frac{T-h-t}{T-h}+\frac{r-|x|}{nL(T-h)}, &\text{if } t\in [0,T-h]\text{ and } r \le |x| \le r+nL(T-t), \\
\frac{T-h-t}{T-h}, &\text{if } t\in [0,T-h]\text{ and } x\in B(0,r), \\
0, & \text{otherwise},
\end{array}
\right.
\] 
we get that for a.e.  $(t,x) \in [0,T-h]\times \mathbb{R}^n$,
\[
\left( |\tau_h W_{\epsilon}-W_\epsilon|^2\partial_t\varphi +\sum_{i=1}^n \Braket{\tau_h W_{\epsilon}-W_\epsilon,B_i (\tau_h W_{\epsilon}-W_\epsilon)}\partial_{x_i}\varphi \right)(t,x)\le 0, 
\]
and then
\begin{multline}
\label{translatee_en_temps}
 \int_0^{T-h} \int_{B(0,r)} |\tau_h W_{\epsilon}-W_{\epsilon}|^2\, d x\, d t \\
\le (T-h)\int_{B(0,r+nL(T-h))} |W_\epsilon(h,x)-W^0(x)|^2\varphi(0,x)\, d x \\
\le T\int_{B(0,r+nLT)} |W_\epsilon(h,x)-W^0(x)|^2\, d x.
\end{multline}
The conclusion then follows from the following result whose proof is very close to that of \cite[Proposition 7]{DLS}.
\begin{lemma}
For all $\xi \in \mathcal{C}_c^{\infty}(\mathbb{R}^n,\mathbb{R}^+)$, one has
\[
\lim_{h\to 0}\sup_{\epsilon>0} \int_{\mathbb{R}^n} |W_\epsilon(h,x)-W^0(x)|^2\xi(x)\, d x = 0.
\]
\end{lemma}

According to Theorem~\ref{theo_simons} and estimate \ref{estim_Weps_LinfL2}, the sequence $(W_\epsilon)_{\epsilon>0}$ admits a subsequence (not relabeled) which  converges strongly in $L^2([0,T];L^2(\omega,\mathbb{R}^m))$ to some $\tilde{W}$ and weakly in $L^2([0,T];L^2(\mathbb{R}^n,\mathbb{R}^m))$ to some $W \in L^2([0,T];L^2(\mathbb{R}^n,\mathbb{R}^m))$. By uniqueness of the limit, we infer that $\tilde{W}=W \in L^2([0,T];L^2(\mathbb{R}^n,\mathbb{R}^m)$. Let us take a test function $\phi \in \mathcal C^\infty_c(]0,T[ \times \mathbb R^n)$ in \ref{cons}. Multiplying this inequality by $\epsilon$ and passing to the limit as $\epsilon \to 0$ yields $W=P_K(W)$ a.e. in $]0,T[ \times \mathbb R^n$ which shows that $W \in L^2([0,T];L^2(\mathbb{R}^n,K))$. Finally, passing to the limit as $\epsilon\to 0$ in \ref{inequa_Weps} shows that $W$ is a solution in the sense of Definition~\ref{defsolsystnonborn} to the problem~\ref{SystNonBorn}. Note finally that, by uniqueness of the solution to \ref{SystNonBorn} (see \cite[Lemma 9]{DLS}), there is no need to extract a subsequence to get the above convergences as $\epsilon \to 0$.
\end{proof}

The construction of the solution $W$ to \ref{SystNonBorn} rests on the assumption that the initial data $W_0 \in H^1(\mathbb{R}^n,K)$.
Let us now explain how to construct a solution $W$ when $W_0$ only belongs to $L^2(\mathbb{R}^n,K)$. We use here the following result whose proof can be found in~\cite{DLS}
\begin{theorem}
\label{regularite_systnonborn}
Let $W^0$ and $\tilde{W^0}\in H^1(\mathbb{R}^n,K)$. We denote by $W$ (resp. $\tilde{W}$)  the solution in $L^2(0,T;L^2(\mathbb{R}^n,K))$ to problem~\ref{SystNonBorn} in the sense of Definition~\ref{defsolsystnonborn} associated with $W^0$ (resp. $\tilde{W^0}$). Then, $W$ and $\tilde W$ belong to $\mathcal{C}([0,T];L^2(\mathbb{R}^n,K))$, and, in addition, we have the following estimate
\[
\forall t\in [0,T],\forall r>0,\quad \left\| W(t,\cdot)-\tilde{W}(t,\cdot)\right\|_{L^2(B(0,r))}\le \left\| W^0-\tilde{W^0}\right\|_{L^2(B(0,r+nLT))},
\]
where $L$ is the maximum of the spectral radii of the matrices $B_i$.
\end{theorem}
By mollification, let us construct a sequence $(W^0_k)_{k\in \mathbb{N}}$ such that $W^0_k \in H^1(\mathbb{R}^n,K)$ for all $k\in \mathbb{N}$, which converges to $W^0$ in $L^2(\mathbb{R}^n,K)$. The estimates of Theorem~\ref{regularite_systnonborn} imply that
\begin{equation}
\label{comparaison_cauchy}
\sup_{t\in [0,T]} \left\| W_k(t,\cdot)-W_l(t,\cdot) \right\|_{L^2(\mathbb{R}^n)} \le \left\| W^0_k-W^0_l\right\|_{L^2(\mathbb{R}^n)},
\end{equation}
where $W_k$ (resp. $W_l$) is the solution to ~\ref{SystNonBorn} associated with the initial condition $W^0_k$ (resp. $W^0_l$). It follows that the sequence $(W_k)_{k \in \mathbb N}$ is of Cauchy type in $L^{\infty}(0,T;L^2(\mathbb{R}^n,\mathbb{R}^m))$, and therefore it converges strongly in $L^{\infty}(0,T;L^2(\mathbb{R}^n,\mathbb{R}^m))$ to some function $W\in L^{\infty}(0,T;L^2(\mathbb{R}^n,\mathbb{R}^m))$. Thanks to the strong convergence, we find that $W$ satisfies the inequality~\ref{inequadefsolsystnonborn}.  In addition, since  $W_k=P_K(W_k)$ for all $k \in \mathbb N$,  we deduce that $W=P_K(W)$ which ensures that $W \in L^{\infty}(0,T;L^2(\mathbb{R}^n,K))$. The following result has thus been established.
\begin{theorem}
Let $W^0 \in L^2(\mathbb{R}^n,K)$, then there exists a unique solution $W\in L^\infty(0,T,L^2(\mathbb{R}^n,K))$ to \ref{SystNonBorn} in the sense of Definition~\ref{defsolsystnonborn}.
\end{theorem}

\section{Conclusion}

In definitive, the relaxed problem~\ref{ApproxSystNonBorn}, that was
used in~\cite{DLS} to derive formally a definition of weak solutions
of hyperbolic constrained problems, is in fact a rigorous way to
construct weak solutions of hyperbolic constrained problems. It is
worth noting that this relaxation procedure is deeply related to
viscoplastic models. In order to fully apply this theory to mechanical
problems, one should consider problems that are posed in bounded
spatial domain. To do so, a new formulation of weak solutions to
Friedrichs' systems posed in bounded domains is proposed in~\cite{MDS},
without constraints. It remains now to investigate the interactions
between the boundary conditions which are considered in~\cite{MDS} and
the convex constraints.

\section*{Acknowledgments}
The authors wish to express their thanks to Bruno Despr{\'e}s for many
stimulating conversations. The first and the third authors would also
mention that they have benefited from the unrivaled working atmosphere
of the 2D24.



\medskip
Received xxxx 20xx; revised xxxx 20xx.
\medskip

\end{document}